\documentclass[a4paper,12pt]{amsart}
\calclayout
\usepackage{hyperref}
\usepackage{amssymb}
\usepackage{amsmath}
\usepackage{amsfonts}
\usepackage[all]{xy}
\usepackage{mathabx}

\newtheorem{Thm}{Theorem}[section]
\newtheorem{cor}[Thm]{Corollary}
\newtheorem{lem}[Thm]{Lemma}
\newtheorem{prop}[Thm]{Proposition}
\theoremstyle{definition}

\theoremstyle{remark}
\newtheorem{Rmk}[Thm]{Remark}
\numberwithin{equation}{section}

\begin{document}

\title{The Conley-Zehnder indices of the Reeb flow action along $S^1$-fibers over certain Orbifolds}

\author{Sokmin Hong}

\address{Department of Mathematical Sciences, Seoul National University, GwanAkRo 1, Gwanak-Gu, Seoul 08826, Korea}

\email{smhong@snu.ac.kr}

\subjclass[2010]{53D12, 57R18, 14A20}

\keywords{Conley-Zehnder index, orbifold Chern class, weighted projective space, Brieskorn polynomial}


\date{}

\dedicatory{}

\commby{}


\begin{abstract}
We prove a useful relation between the Conley-Zehnder indices of the Reeb vector flow action along periodic orbits in  prequantization bundles and the orbifold Chern class of the base symplectic orbifolds motivated by the well-known case of manifolds. We also apply this method to primary examples.
\end{abstract}

\maketitle

\section{Introduction}

The Conley-Zehnder index is a crucial concept in researching the symplectic (co)homology and other related topics. For example, Otto van Koert introduced the notion of mean Euler characteristic in \cite{O3} as
$$
\chi_m(W) \,=\,
\frac 12
\left(
\liminf_{N \to \infty}  \frac{1}{N} \sum_{i=-N}^N (-1)^i b_i(W) \,+\,
\limsup_{N \to \infty}  \frac{1}{N} \sum_{i=-N}^N (-1)^i b_i(W)
\right),
$$
with a simply connected Liouville filling $\left(W,d\lambda\right)$, and also provided a simple and useful formula
$$
\chi_m(W)=\frac{\sum_{i=1}^k (-1)^{\mu(\Sigma_{T_i})-\frac{1}{2}\dim (\Sigma_{T_i}/S^1) } \phi_{T_i;T_{i+1},\ldots T_k} \chi^{S^1}(\Sigma_{T_i})}{|\mu_P|},
$$
where $\mu_p$ is the Conley-Zehnder index of a Reeb vector flow action along a principal orbit in \cite{O2}. 

The terms `Conley-Zehnder index' and `Maslov index', however, seem to be being used for varied forms of definitions involved in paths or loops of symplectic matrices, and sometimes be used even interchangeably. So, we stick with the definitions in \cite{S} for them in the sequel. The most important property of the Conley-Zehnder index $\mu_{CZ}$ in this article is the so-called loop property, which goes as follows :
$$\mu_{CZ}\left(\phi\psi\right)=\mu_{CZ}\left(\psi\right)+2\mu\left(\phi\right),$$
where $\psi$ is a path of symplectic matrices with certain properties and $\mu\left(\phi\right)$ is the Maslov index of the loop $\phi$, i.e. $\mu\left(\phi\right)=\text{deg}\left(\rho\circ\phi\right).$ Here, the integer-valued function $\rho$ on the set of all symplectic matrices requires an amount of work to define, but it is simply the complex determinant when the symplectic matrix is as special as to be a unitary matrix.

Meanwhile, there is a well-known relation between the Conley-Zehnder indices of the Reeb vector flow along periodic orbits in prequantization bundles and the first Chern classes of base symplectic manifolds. For example, Otto van Koert showed one in \cite{O1} under the following assumption: for the class of integral symplectic manifolds $\left(Q,\omega\right)$ satisfying
\begin{enumerate}
\item[(i)] $[\omega]$ is primitive
\item[(ii)] $c_1(Q) = c [\omega]$ for some $c\in\mathbb Z$
\item[(iii)] $Q$ is simply connected.
\end{enumerate}

In Lemma 3.3 of \cite{O1}, it was shown that the Conley-Zehnder index of the Reeb vector flow action is
$$\mu(\gamma)=2c,$$
along a circle fiber $\gamma=\pi^{-1}(q)$ of the prequantization bundle $P$ over $\left(Q,\omega\right)$.

This article has almost the parallel consequence as that of Lemma 3.3 of \cite{O1} under the orbifold setting. That is, a symplectic manifold should be replaced by a symplectic orbifold, and the orbifold (co)homology or the orbifold Chern class plays a role in place of their counterparts in the usual sense.

In Section 3, we prove the main theorem in the orbifold case by a similar method used in Lemma 3.3 of \cite{O1}. Even though the ingredient idea in Lemma 3.3 of \cite{O1} works almost the same way even we switch the situation to the orbifold setting, there are lots of subtleties that need to be tackled carefully because of it being an orbifold in lieu of being just a manifold. First of all, every Reeb orbit over a symplectic orbifold does not have the identical period in contrast to the manifold case. For this reason, we will overview basic facts about general orbifolds in Section 2 to help understand those unparallel situations precisely. Despite similarities with the manifold cases, not only does the consequence of this article provide us larger scope of applications but also we can enjoy interesting features of orbifolds while dealing with orbibundles, classifying spaces and their examples and so forth.

The algebraic varieties of weighted projective spaces and their complete intersections are well-known examples as orbifolds with special properties. A great number of mathematicians have been and will be researching them using a variety of mathematical tools. In Section 4, we show that our main theorem works for those spaces and compute the actual values regarding them defined in the previous section.

Although the main theorem works most effectively for principal orbits, it is still quite useful even for non-principal ones as well, and we'll see it in the final section.

Additionally, it was shown that
\begin{equation}\label{OldBrieskorn}
\mu_P=2~\text{lcm}\{a_j\}\left(\sum_{j=0}^n\frac{1}{a_j}-1\right),
\end{equation}
in the case of the Brieskorn polynomial of
$$\sum_{j=0}^n z_j^{a_j}=0,$$
 in \cite{O2}. The reader should be warned that they used the standard complex structure of $\mathbb C^{n+1}$ where the Brieskorn manifold resides as an affine hypersurface while computing the indices. In this article, however, we can't use that method as it is because every structure involved here must be invariant under the actions in order to be orbifold objects.

\section{rudiments on orbifolds}
The easiest way to define an orbifold is by using local uniformizing systems. The exact definition can be found in \cite{ARL} together with basic properties of general orbifolds. To make a long story short, a local uniformizing system is a triple $\left(\tilde U,\Gamma,\varphi \right)$, where $\tilde U$ is a connected open subset of $\mathbb R^n$, $\Gamma$ is a finite group acting on $\tilde U$, $\varphi$ is a map from $\tilde U$ to the base space invariant under $\Gamma$. We usually denote an orbifold by $$ \mathcal X= \left(X,\mathcal U\right),$$
where $X$ is the base space and $\mathcal U$ is its orbifold atlas.

In the special case when
\begin{enumerate}
\item $\tilde U\cong \mathbb C^n$
\item $\Gamma$ is a finite subgroup of $GL(n,\mathbb C)$
\item all the embeddings are holomorphic,
\end{enumerate}
then we call such an orbifold $\mathcal X$ a complex orbifold. In this case, the base space $X$ is a complex space with quotient singularities at worst \cite[p.123]{BG}.

The orbifold canonical divisor $\mathcal K_{\mathcal X}^{orb}$ of $\mathcal X$ is a different notion from
$K_{X}$,
the canonical divisor of $X$ as a complex space.

Let us consider the following example: Suppose $S^1=\{\zeta\in\mathbb C:\left|\zeta\right|=1\}$ acts on $S^3=\{\left(z_0, z_1\right)\in\mathbb C^2:\left|z_0\right|^2+\left| z_1\right|^2=1\}$ by
$$\zeta\cdot\left(z_0, z_1\right)=\left(\zeta z_0,\zeta^m z_1\right), \qquad\zeta\in S^1,$$
where $m$ is an integer bigger than 1. Then any orbit
$$\left\{\zeta\cdot\left(z_0, z_1\right):\zeta\in S^1\right\}, \qquad z_0\neq0$$
is principal whereas it is not when $z_0$ equals zero. Actually, the quotient space $X$ under this action turns out to be a weighted projective space $\mathbb P(1,m)$ and with a suitable orbifold structure it becomes an orbifold, which we denote by $\mathcal X$ as an orbifold. By the way, $\mathbb P(1,m)$ is known to be isomorphic to $\mathbb {CP}^1$ as an algebraic variety just as any other single-dimensional weighted projective spaces.

To see its orbifold structure with more accuracy, consider $$\mathbb {CP}^1\cong S^2$$ as a complex space, and take a small open disc $D$ centered at the north pole $p=[0:1]$. Then, define $\varphi : D\longrightarrow D$ by
$$w=\varphi(z)=z^m \qquad \text{and} \qquad \Gamma\cong\mathbb Z_m.$$
Clearly on $D$, we have
$$K_X=dw \qquad \text{and} \qquad \mathcal K_{\mathcal X}^{orb}=dz,$$ which is different from
$$\varphi^\ast K_X=z^{m-1}dz.$$
They are unequal because every element in $\Gamma$ is a reflection, which doesn't produce genuine singularities (refer to Theorem 4.4.1 in \cite{BG}). In other words, even though the north pole is an orbifold singular point due to its nontrivial isotropy group, it is not singular in the usual sense because $\mathbb {CP}^1$ is very smooth all over.
Here, the term $z^{m-1}$ can be interpreted as a divisor $\left(1-1\slash m \right)[p]$, which allows us to say that its {\sl orbifold Chern number} is $2-\left(1-1\slash m\right)=1+1\slash m$.

Now take a look at the "tautological line orbibundle" $\pi:S^3\rightarrow\mathcal X$, whose fiber over $\left[z_0:z_1\right]\in\mathcal X$ is $\left\{\zeta\cdot\left(z_0, z_1\right)=\left(\zeta z_0, \zeta^m z_1\right)\in S^3:\zeta\in S^1\right\}\cong S^1$. Although it is an orbifold $S^1$-bundle since its pull-back onto the orbifold chart becomes an ordinary $S^1$-bundle that is compatible with our orbifold structure, it is not a fiber bundle in the usual sense because we may not get an $S^1$-trivialization about the north pole. Otherwise, let's suppose $\Phi : U\times S^1\cong \pi^{-1}(U)$ is an $S^1$-trivialization, where $\{\text{the north pole}\}\in U$. Then, $\Phi\left(\{x\}\times \left\{e^{2\pi i\slash m}\right\}\right)$ and $\Phi\left(\{x\}\times \left\{1\right\}\right)$ converges to the same point in the fiber over the north pole as $x$ approaches the north pole, which is absurd. However, $m$-times of this orbibundle becomes a genuine circle bundle (refer to Proposition 4.4.22 in \cite{BG}).

Next, the {\sl orbifold (co)homology groups} and {\sl orbifold homotopy groups}
 of $\mathcal X$ is defined as the usual (co)homology groups and homotopy groups of the classifying space $B\mathcal X$ of $\mathcal X$. Here, $p:B\mathcal X \rightarrow X$ is a singular fibration over $\mathcal X$, whose singular fiber over the north pole is the Eilenberg-Maclane space $K\left(\mathbb Z_m,1\right)$ and the generic fiber is the infinite sphere except over the north pole. $H^{orb}_2(\mathcal X,\mathbb Q)$ equals to $H_2(X,\mathbb Q)\cong\mathbb Q$ because $\mathbb Q$ is a field (see Corollary 4.3.8, \cite{BG}). To understand the natural projection $p_\ast: H^{orb}_2(\mathcal X,\mathbb Q) \rightarrow H_2( X,\mathbb Q)$, think of $B\mathcal X$ made by attaching the boundary of a disk with degree $m$ along the boundary of a small puncture on the north pole of the sphere. Then we can easily see that $p_\ast$ induces the division by $m$ in $\mathbb Q$.
 
Getting the orbifold (co)homology groups in $\mathbb Z$-coefficient is a little tricky. Consider the long exact sequence of the orbifold homopoty groups for an orbibundle $\mathcal P$ over $\mathcal X$ with the fiber $F$ (\cite[Theorem 4.3.18]{BG}):
\begin{equation}\label{exactseq}
\cdots\rightarrow\pi_{n}^{orb}\left(\mathcal P\right)\rightarrow\pi_{n}^{orb}\left(\mathcal X\right)\rightarrow\pi_{n-1}^{orb}\left(F\right)\rightarrow\cdots.
\end{equation}
Now that our fiber $F$ is $S^1$, it is easy to see $\pi_1(B\mathcal X)=0,\pi_2(B\mathcal X)=\mathbb Z$, and hence $H_{1}(B\mathcal X)=0,H_{2}(B\mathcal X)=\mathbb Z$. Now, let $U$ be an open subset in $B\mathcal X$ that is the upper disk containing the north pole. i.e. it is a singular fiber bundle over a disk whose generic fiber is contractible and singular fiber is $K(\mathbb Z_m,1)$. Let $V$ be the lower disk. i.e. it is the trivial bundle over a disk with a contractible fiber, and hence a contractible space. Then $U\cap V$ is a bundle over $S^1$ with a contractible fiber, and hence homotopy equivalent to $S^1$.
Therefore,
$$\begin{aligned}
&H_\ast(U) = H_\ast(K(\mathbb Z_m,1)),\\
&H_\ast(V) = H_\ast(\{pt\}),\\
&H_\ast(U\cap V) = H_\ast(S^1).
\end{aligned}$$
From the Mayer-Vietoris sequence, we conclude that
$$H_q^{orb}(\mathcal X) = \begin{cases}\mathbb Z& q=0,2\\\mathbb Z_m&q>1 \text{ odd}\\0&q=1\text{ or }q>2 \text{ even}\end{cases}.$$
Similarly, we can get the cohomology groups :
$$H^q_{orb}(\mathcal X) = \begin{cases}\mathbb Z& q=0,2\\\mathbb Z_m&q>2 \text{ even}\\0&q>0 \text{ odd}\end{cases}.$$

There's one last thing to point out about orbifolds. On a prequantization bundle over a manifold, every Reeb flow rotates with the same period. i.e. The $S^1$-action is free. But this is not the case over an orbifold, because the $S^1$-action is only locally free. So, for a prequantization bundle $\left(M, \xi,\eta\right)$ over an orbifold, we call a periodic orbit $\gamma_P$ of $\xi$ is {\sl principal} if $\gamma_P$ has the longest period among all the periodic orbits of $\xi$ when it exists.

\section{The main theorem}

The main theorem will be proved in this section. According to Theorem 7.1.6 in \cite{BG}, an integral almost K\"ahler orbifold admits a circle orbibundle generated by its symplectic form, whose total space becomes a K-contact orbifold. With counting the action by the Reeb vector flow along a fiber as a path of symplectomorphisms, we will compute its Conley-Zehnder index. The proof will be an orbifold version analogous to \cite{O1} with additional considerations.
\begin{Thm}\label{mainthm}
Let $\left(\mathcal Z, \omega \right)$ be a Hodge orbifold, so that it admits an $S^1$-orbibundle $\pi:M\rightarrow\mathcal Z$ whose total space $M$ has a K-contact structure $(\xi,\eta,\Phi,g)$ where $d\eta=\pi^\ast\omega$.
Further, if

\begin{enumerate}
\item[(i)] $c_1^{orb}\left(T\mathcal Z\right)=-b_{\mathcal Z}\left[\omega\right]\in H^2\left(\mathcal Z,\mathbb Q\right)$
for some integer $b_{\mathcal Z}\in\mathbb Z$
\item[(ii)] $\pi_1^{orb}\left(\mathcal Z\right)=0$
\item[(iii)] $M$ is a manifold,
\end{enumerate}
then the Conley-Zehnder index $\mu_{CZ}$ of the Reeb vector flow action along an orbit $\gamma$ wound $\left|\Gamma_q\right|$-times is
$$\mu_{CZ}\left(\left|\Gamma_q\right|\cdot\gamma\right)=2b_{\mathcal Z},$$
where $q$ is the image of $\gamma$ under $\pi$ and $\Gamma_q$ is the isotropy group at $q$. For a principal orbit, we know that $\left|\Gamma_q\right|=1$ and denote by $\mu_p\left(\mathcal Z\right)$ its Conley-Zehner index.
\end{Thm}

\begin{proof}
Before proceeding to prove the theorem, let us clarify what the path $\phi$ of symplectic matrices is whose Conley-Zehnder index we seek for.
For a point $x$ in $\mathcal Z$, let $\gamma_x$ be the $S^1$-fiber on $x$ or the periodic Reeb orbit, and for each point $t\in\gamma_x$, let $Q^x_t$ be the horizontal space at $t$. First, fix any point in $\gamma_x$ as 0, and take any symplectomorphism $\Phi:\mathbb C^n \rightarrow Q_0^x$. Now let $\alpha^x(t):Q^x_0\longrightarrow Q^x_t$ be a symplectomorphism induced by the Reeb flow action from 0 to $t$ and $\beta^x(t)=\left(d\pi|_{Q^x_t}\right)^{-1}\circ d\pi|_{Q^x_0}:Q^x_0\longrightarrow Q^x_t$. Then, $\phi(t)$ is defined to be $\Phi^{-1}\circ\alpha^x(t)^{-1}\circ\beta^x(t)\circ\Phi$. Of course, different $\Phi$ may change the path $\phi$, but keeps the same Conley-Zehnder index. Also, if the orbit $\gamma$ happens to be principal, then $\phi$ becomes a loop.

Write $B\mathcal Z\xrightarrow{p}\mathcal Z$ for the classifying space of $\mathcal Z$. Choose points $q\in\mathcal Z$, $\tilde q\in T\mathcal Z$ such that $p\left(\tilde q\right)=q$ and write the periodic Reeb orbit over $q$ by $\gamma_q$. As shown in \cite{O1}, the Conley-Zehnder index along $\gamma_q$ can be computed by considering the Maslov index along $\gamma_q$ of the contact structure $\mathcal D$ of $\eta$ as a symplectic bundle.

Due to Theorem 4.3.11 in \cite{BG}, we can construct a $U(1)$-bundle $\tilde M\xrightarrow{\tilde\pi}B\mathcal Z$ over $B\mathcal Z$ corresponding to $M$, as in the following diagram
$$\xymatrix{\tilde M \ar[d]_{\tilde\pi}\ar[rr]^{\tilde p} && M \ar[d]^{\pi} \\ B\mathcal Z \ar[rr]_p&& \mathcal Z.}$$ 
Also, consider the vector bundle $\tilde T\mathcal Z$ over $B\mathcal Z$ corresponding to the tangent bundle $T\mathcal Z$ over $\mathcal Z$, as in the diagram
$$\xymatrix{\tilde T\mathcal Z \ar[d] && T\mathcal Z \ar[d] \\ B\mathcal Z \ar[rr]_p && \mathcal Z.}$$
Then, $\tilde M, \tilde T\mathcal Z$ are generic bundles over $B\mathcal Z$ and moreover, $p^\ast c_1^{orb}\left(T\mathcal Z\right)\in H_{orb}^2\left(\mathcal Z,\mathbb Z\right)$ is the first Chern class of $\tilde T\mathcal Z$.

Due to the assumption that
$\pi_1^{orb}\left(\mathcal Z\right)=0$,
we know that
$$\pi_2^{orb}\left(\mathcal Z\right)=H_2^{orb}\left(\mathcal Z,\mathbb Z\right) \qquad \text{and} \qquad
H_{orb}^2\left(\mathcal Z,\mathbb Z\right)=H_2^{orb}\left(\mathcal Z,\mathbb Z\right)^\ast$$
by the universal coefficient theorem. Considering a primitive element that generates $[p^\ast\omega]$ and $c_1^{orb}\left(T\mathcal Z\right)$, we may pick a sphere $$\iota : S\rightarrow B\mathcal Z,$$ in $H_2\left(B\mathcal Z,\mathbb Z\right)$ which is
$$\left<\left[p^\ast\omega \right],\iota(S)\right>=k\in\mathbb Z^+.$$
Such $S$ does exist because the Betti part of $H_2\left(B\mathcal Z,\mathbb Z\right)$ is finite dimensional \cite[Proposition 7.2.3]{BG}.

Consider the pull-back bundle
$$P=\xymatrix{\iota^\ast\tilde M \ar[d]_{\tilde\pi} \\ S}$$
over $S$. Note that $\left(\iota\tilde p\right)^\ast\eta$ gives it a connection and in order to get its curvature, compute
$$
\tilde\pi^\ast\iota^\ast p^\ast\omega=\iota^\ast\tilde\pi^\ast p^\ast\omega
=\iota^\ast\tilde p^\ast \pi^\ast \omega
=\iota^\ast\tilde p^\ast d\eta.
$$
Therefore, the first Chern class of $P$ is
$$c_1(P)=\left[\iota^\ast p^\ast\omega\right].$$
Now, define the complex line bundle $L$ by
$$L=P\times_{S^1}\mathbb C$$
over $S$, which is isomorphic to $\mathcal O(k)$ because of
$$\left<c_1(P),S\right>=\left<\left[\iota^\ast p^\ast\omega \right],S\right>=\left<\left[p^\ast\omega \right],\iota(S)\right>=k.$$ 
Hence we can take a section $\sigma : S \longrightarrow L$ that vanishes only at one point $\tilde q$ with multiplicity $k$ and then extend it to a continuous map
$$\bar\sigma : D^2\longrightarrow P.$$

Since $\pi$ maps $\mathcal D$ symplectomorphically to $T_q\mathcal Z$ on every point of $\pi^{-1}(q)$, $\tilde p^\ast\mathcal D$ is mapped symplectomorphically to $\tilde T_{\tilde q}\mathcal Z$ via $\tilde\pi$ on every point of $\tilde\pi^{-1}\left(\tilde q\right)$. Therefore, the Maslov index along $\gamma_q$ of $\mathcal D$ is the same as the Maslov index along $\tilde\gamma_{\tilde q}=\tilde p^{-1}\left(\gamma_q\right)$ of $\tilde p^\ast\mathcal D$, so that the technique introduced in \cite{O1} for manifolds is still effective. 

To be more precise, by considering  trivializations of $\iota^{\ast}\tilde T\mathcal Z$ over a splitting $S=S\backslash\left\{\tilde q\right\}\cup\left\{\tilde q\right\}$ and its overlap map, we will relate the Maslov index along $\tilde\gamma_{\tilde q}$ with $p^\ast c_1^{orb}\left(T\mathcal Z\right)$ as done in \S2.6 of \cite{MS}. Let us assume $k=1$ for a moment. Since $S\backslash\left\{\tilde q\right\}$ is contractible, we may choose trivializations $\Phi:S\backslash\left\{\tilde q\right\}\times\mathbb C^n\cong\tilde T\mathcal Z|_ {S\backslash\left\{\tilde q\right\}}$, and $\Psi:S\backslash\left\{\tilde q\right\}\times S^1\cong P|_{S\backslash\left\{\tilde q\right\}}$. Write $\Psi^{-1}\left(\bar\sigma\left(\tilde x\right)\right)=\{\tilde x\}\times\left\{t_{\tilde x}\right\}\in S\backslash\left\{\tilde q\right\}\times S^1$, and identify $\Theta_{\tilde x}:\tilde T_{\tilde x}Z \cong Q^{x}_{0}$ through $d\pi\circ p$, for $\tilde x\in S\backslash\left\{\tilde q\right\}$, $x=p\left(\tilde x\right)$. 
Then by covering all $\tilde x\in S\backslash\{\tilde q\}$, 
$$\left(\tilde p^\ast\right)^{-1}\circ\alpha^{x}\left(t_{\tilde x}\right)\circ\Theta_{\tilde x}\circ\Phi$$ gives a trivialization of $\tilde p^\ast\mathcal D$ over $\bar\sigma\left(S\backslash\left\{\tilde q\right\}\right)$.

For the part of $\tilde T_{\tilde q}Z$, the horizontal lift along $\gamma_{q}$ will work. Then, the overlap map with this splitting and trivializations is exactly the path of symplectomorphisms we pursue. In fact, this description requires a bit more precise verifications, which can be found in the proof of Lemma 3.3 of \cite {O1}. 

For general $k$, have the boundary of $D^2$ map to the fiber $\tilde\pi^{-1}\left(\tilde q\right)$ winding it $k$ times, so that $\bar\sigma(D^2)$ plays a roll of the (upper) capping disk along $\tilde\gamma_{\tilde q}$ with degree $k$. Also, attach the trivial symplectic vector bundle over a disk with the boundary sphere winding only once along $\tilde\gamma_{\tilde q}$ for the (lower) capping disk to make a sphere. In the end, by the nature of the Maslov index, we have
\begin{align*}
k\cdot\mu_{CZ}\left(\tilde\gamma_{\tilde q}\right)=&2\left<c_1\left(\iota^\ast\tilde T\mathcal Z\right),S\right>\\
=&2\left<\iota^\ast c_1\left(\tilde T\mathcal Z\right),S\right>\\
=&2\left<c_1\left(\tilde T\mathcal Z\right),\iota(S)\right>\\
=&2\left<p^\ast c_1^{orb}\left(T\mathcal Z\right),\iota(S)\right>\\
=&2b_{\mathcal Z}\cdot\left<\left[p^\ast\omega \right],\iota(S)\right>\\
=&2k\cdot b_{\mathcal Z}.
\end{align*}

Now that the degree of $\tilde p$ is equal to $\left|\Gamma_q\right|$, the order of the isotropy group of $q$ in the base space, we get 
$\mu_{CZ}\left(\tilde\gamma_{\tilde q}\right)=\mu_{CZ}\left(\left|\Gamma_q\right|\cdot\gamma_{q}\right),$
which lead us to the conclusion.
\end{proof}

\begin{Rmk}
As for the condition (ii), recall (\ref{exactseq}) the exact sequence of orbifold homotopy groups. If $M$ happens to be a simply connected manifold, we have
$$\pi_{1}(M)=\pi_{1}^{orb}(M)=0,$$
and hence it follows that
$$\pi_{1}^{orb}(\mathcal Z)=0.$$
\end{Rmk}

\section{The weighted projective spaces and their complete intersections}

We denote by $\mathbb P\left({\bf w}\right)$ the weighted projective space with weights
$${\bf w}=\left(w_0,w_1,\cdots,w_n\right),$$ and we assume additionally
$$\gcd\left(w_0,w_1,\cdots,w_n\right)=1.$$
Let us use the following notations to make it easy in dealing with weighted projective spaces:
$$\left|{\bf w}\right|=\sum_{j=0}^{n}w_j,$$
$$\left\|{\bf w}\right\|=\prod_{j=0}^{n}w_j,$$
$$d_j=\gcd\left(w_0,\cdots,\widehat{w_j},\cdots,w_n\right),$$
$$e_j=\text{lcm}\left(d_0,\cdots,\widehat{d_j},\cdots,d_n\right),$$
$$a_{\bf w}=\text{lcm}\left(d_0,\cdots,d_n\right),$$
$${\bf\bar w}=\left(\frac{w_0}{e_0},\cdots,\frac{w_n}{e_n}\right).$$
Note that
${\bf w}={\bf\bar w}$ if and only if $d_j=1$ for all $j$ or if and only if $a_{\bf w}=1$.
We call $\mathbb P\left({\bf w}\right)$ is {\sl well-formed} in those cases.

In fact, an algebraic variety $\mathbb P\left({\bf w}\right)$ can have different orbifold structures, but in this article, we only consider the structure as a quotient space where $S^1$ acts locally free on $S^{2n+1}$. This well-known $S^1$-action is the generalization of the one introduced in Section 2:
$$\zeta\cdot\left(z_0, z_1,\cdots,z_n\right)\mapsto\left(\zeta^{w_0} z_0,\zeta^{w_1} z_1,\cdots,\zeta^{w_n}z_n\right), \zeta\in S^1.$$
The orbifold chart to define this orbifold structure can be found in p.143 \cite{M}. Roughly;
For $j\in\{0,\cdots,n\}$, let $U_j$ be the subset of $\mathbb P\left({\bf w}\right)$ such that
$$U_j:=\left\{\left[z_0:\cdots:z_n\right]~|~z_j\neq0\right\},$$
and $\tilde U_j$ the set of points in $\mathbb C^{n+1}\backslash\{\bf 0\}$ such that $y_j=1$. Then, $\tilde U_j$ is the local uniformizing chart with
$$\varphi_j : \left(y_0,\cdots,1_j,\cdots,y_n\right) \in \tilde U_j \mapsto \left[y_0:\cdots:1_j:\cdots:y_n\right] \in U_j.$$
Its uniformizing group $\Gamma_j$ is
$$\mu_{w_j}\cong\mathbb Z_{w_j},$$
where $\zeta\in\Gamma_j$ acts on $\tilde U_j$ by
$$\zeta\cdot\left(y_0,\cdots,1_j,\cdots,y_n\right)\mapsto\left(\zeta^{w_0}y_0,\cdots,1_j,\cdots,\zeta^{w_n}y_n\right).$$

We don't need to mention embeddings between local uniformizing charts, because they are not used in the sequel. 
We denote $\mathbb P\left({\bf w}\right)$ with this orbifold structure by $\mathcal P\left({\bf w}\right)$.

\begin{Rmk}
There is another orbifold structure in $\mathbb P\left({\bf w}\right)$:
The finite set $G_{\bf w}=\mathbb Z_{w_0}\times\cdots\times\mathbb Z_{w_n}$ acts on $\mathbb C\mathbb P^n$ to produce a developable complex orbifold $\mathbb {CP}^n\slash G_{\bf w}$. Refer to \cite{BR} for this orbifold structure.
\end{Rmk}

Actually, $\mathcal P({\bf w})$ is a symplectic orbifold: Consider the contact form in $\mathbb C^{n+1}$
$$\eta_{\bf w}=\frac{(2\pi)^{-1}\eta_0}{\sum_{j=0}^{n}w_j\left(\left(x_j\right)^2+(y_j)^2\right)},$$
where $$\eta_0=\frac{i}{2}\sum_{j=0}^{n}\left(z_jd\bar z_j-\bar z_jdz_j\right)=\sum_{j=0}^{n}\left(x_jdy_j-y_jdx_j\right).$$
Its Reeb vector field is
$$\xi_{\bf w}=2\pi i\sum_{j=0}^{n} w_j\left(z_j\frac{\partial}{\partial z_j}-\bar z_j\frac{\partial}{\partial \bar z_j}\right)=2\pi\sum_{j=0}^{n} w_j\left(x_j\frac{\partial}{\partial y_j}-y_j\frac{\partial}{\partial x_j}\right),$$
whose integral curve is
$$\left(e^{2\pi w_0it}z_0,\cdots, e^{2\pi w_jit}z_j,\cdots e^{2\pi w_nit}z_n\right).$$
Since $\eta_{\bf w}$ is invariant under $S^1$-action that defines $\mathcal P({\bf w})$, we can use  $d\eta_{\bf w}$ as an orbifold symplectic form $\omega$ for it.

Now, we show $\mathcal P\left({\bf w}\right)$ fits for Theorem \ref{mainthm} and get relevant values.

\begin{lem}
The above symplectic form $[\omega]$ amounts to $-1 \slash \left\|{\bf w}\right\|$ in $H^2\left(\mathbb P\left({\bf w}\right),\mathbb Q\right)\cong\mathbb Q$.
\end{lem}

\begin{proof}
Consider the map
$$f_{\bf w} : \left[z_0;\cdots;z_n\right] \in \mathbb P^n \mapsto \left[z_0^{w_0};\cdots;z_n^{w_n}\right] \in \mathbb P\left({\bf w}\right),$$
whose degree is $\left\|{\bf w}\right\|\slash\gcd\{w_j\}$ in general \cite[Remark 3.5]{M}.
By the following commutative diagram
$$\xymatrix{\mathbb P^2\ar[d]_{f_{\left(w_0,w_1\right)}}\ar[rr]^\iota&&\mathbb P^n \ar[d]^{f_{\bf w}}\\ \mathbb P\left(w_0,w_1\right)\ar[rr]_{\iota_w}&&\mathbb P\left({\bf w}\right),}$$
where $$\iota,\iota_w:\left[z_0;z_1\right]\mapsto\left[z_0;z_1;0;\cdots;0\right],$$ and $\deg\iota=1,~\deg f_{(w_0,w_1)}=w_0w_1\slash\gcd\left(w_0,w_1\right),$ we see that $\deg\iota_w=\frac{\left\|{\bf w}\right\|\cdot\gcd\left(w_0,w_1\right)}{w_0w_1}$.
Also,
$$\iota_w^\ast\eta_{\bf w}=\frac{i}{4\pi}\frac{\bar z_0dz_0-z_0d\bar z_0+\bar z_1dz_1-z_1d\bar z_1}{w_0|z_0|^2+w_1|z_1|^2}.$$
Now, use the chart $z\mapsto\left[z;1;0\cdots;0\right]$ in $\mathbb P\left(w_0,w_1\right)$, so that
$$\begin{aligned}
\iota_w^\ast d\eta_{\bf w}=&-\frac{i}{4\pi}\frac{\left(2w_1dz\wedge d\bar z\right)}{\left(w_0 |z|^2+w_1\right)^2}\\
=&-\frac{i}{2\pi}\frac{\left(-2i w_1dx\wedge dy\right)}{\left(w_0 |z|^2+w_1\right)^2}\\
=&-\frac{1}{\pi}\frac{\left(w_1rdr\wedge d\theta\right)}{\left(w_0 r^2+w_1\right)^2}.
\end{aligned}$$
The integral is, then
$$-\frac{1}{\pi}\int_{0}^{2\pi}\int_{0}^{\infty}\frac{w_1r}{\left(w_0r^2+w_1\right)^2}drd\theta=-\frac{1}{2\pi}\int_{0}^{2\pi}\int_{w_1}^{\infty}\frac{w_1}{w_0s^2}dsd\theta
=-\frac{1}{w_0}.$$
Since the uniformizing group of this chart is of order $w_1\slash\gcd\left(w_0,w_1\right)$, the integral value should be $-\frac{\gcd\left(w_0,w_1\right)}{w_0w_1}$.
By factoring in the degree of $\iota_{w}$, we get
$$\left<\omega,S^2\right>=-\frac{1}{\left\|{\bf w}\right\|}.$$
\end{proof}

\begin{lem}
The isomorphism $$p^\ast:H^2\left(\mathbb P\left({\bf w}\right);\mathbb Q\right)\longrightarrow H^2\left(B\mathcal P\left({\bf w}\right);\mathbb Q\right)$$ induced by the classifying space map $p:B\mathcal P\left({\bf w}\right)\rightarrow\mathcal P\left({\bf w}\right)$ is the multiplication by $\left\|{\bf w}\right\|$.
\end{lem}
\begin{proof}
Let
$${\bf w_1}=\left(1,w_0,w_1,\cdots,w_n\right)$$
be another weight vector.
Consider the following diagram
$$\xymatrix{\mathbb P^n \ar[d]_{f_{\bf w}}\ar[rr]^\iota && \mathbb P^{n+1} \ar[d]^{f_{\bf w_1}}\\ \mathbb P\left({\bf w}\right) \ar[rr]_{\iota_w} && \mathbb P\left({\bf w_1}\right)}$$
where $$\iota,\iota_w:\left[z_0:z_1:\cdots:z_n\right]\mapsto\left[0:z_0:z_1:\cdots:z_n\right],$$
and $f_{\bf w}, f_{\bf w_1}$ are defined as in Lemma 4.2.
Because of
$$\deg\iota=1,~ \deg f_{\bf w}=\deg f_{\bf w_1}=\left\|{\bf w}\right\|$$ \cite[Remark 3.5]{M}, it follows that
$$\deg\iota_w=1.$$
Since $\iota$ is a good orbifold map \cite[Proposition 3.3]{M}, there is a map $B\iota$ between the classifying spaces induced by $\iota$ \cite[p.119]{BG}
$$\xymatrix{B\mathcal P\left({\bf w}\right) \ar[d]_{p}\ar[rr]^{B\iota} && B\mathcal P\left({\bf w_1}\right) \ar[d]^{p_{1}} \\ \mathcal P\left({\bf w}\right) \ar[rr]_{\iota_w} && \mathcal P\left({\bf w_1}\right).}$$

Meanwhile, we see the following exact sequence
$$\cdots\rightarrow\pi_{2}^{orb}(S^{2n+1})\rightarrow\pi_{2}^{orb}\left(\mathcal P\left({\bf w}\right)\right)\rightarrow\pi_{1}^{orb}\left(S^1\right)\rightarrow\pi_{1}^{orb}\left(S^{2n+1}\right)\rightarrow\cdots,$$
which means $$\pi_{2}^{orb}\left(\mathcal P\left({\bf w}\right)\right)=\mathbb Z.$$
Also, due to $\pi_{1}^{orb}\left(\mathcal P\left({\bf w}\right)\right)=0$,
we have
$$H_2^{orb}\left(\mathcal P\left({\bf w}\right)\right)= H_2^{orb}\left(\mathcal P\left({\bf w_1}\right)\right)=\mathbb Z,$$
and hence
$$H_{orb}^2\left(\mathcal P\left({\bf w}\right)\right)= H_{orb}^2\left(\mathcal P\left({\bf w_1}\right)\right)=\mathbb Z.$$

Let us regard the classifying space as a singular fibration \cite[p.117]{BG}. Then, it is clear that $B\iota$ is an embedding in the usual sense, and that the map $\alpha$ in the following diagram must be the identity map.
$$\xymatrix{\pi_{2}^{orb}\left(\mathcal P\left({\bf w}\right)\right) \ar[d]\ar[rr]^{B\iota_\ast} && \pi_{2}^{orb}\left(\mathcal P\left({\bf w_1}\right)\right) \ar[d] \\ \pi_1\left(S^1\right) \ar[rr]_{\alpha} && \pi_1\left(S^1\right)}$$
Consequently, $$B\iota^\ast : H_{orb}^2\left(\mathcal P\left({\bf w}\right)\right)\longrightarrow H_{orb}^2\left(\mathcal P\left({\bf w_1}\right)\right)$$ is the identity map.
Therefore $$\deg p_1=\deg p,$$
which allows us to assume that $w_0=1$.

Now, take a look at the following diagrams
$$\xymatrix{\mathbb P^2 \ar[d]_{f_{(w_0,w_1)}}\ar[rr]^\iota && \mathbb P^n \ar[d]^{f_{\bf w}} \\ \mathbb P\left(w_0,w_1\right) \ar[rr]_{\iota_w} && \mathbb P\left({\bf w}\right),}$$
where $$\iota,\iota_w:\left[z_0;z_1\right]\mapsto\left[z_0;z_1;0;\cdots;0\right],$$
and
$$\xymatrix{B\mathcal P\left(w_0,w_1\right) \ar[d]_{p_{1}}\ar[rr]^{B\iota} && B\mathcal P\left({\bf w}\right) \ar[d]^{p}\\ \mathcal P\left(w_0,w_1\right) \ar[rr]_{\iota_w} && \mathcal P({\bf w}).}$$
In this case, we have
$$\deg\iota=1,~\deg f_{\left(w_0,w_1\right)}=w_1\slash1=w_1,$$
and hence
$$\deg\iota_w=\left\|{\bf w}\right\|\slash w_1.$$
By the similar argument, we see that
$$B\iota^\ast : H_{orb}^2\left(\mathcal P\left({\bf w}\right)\right)\longrightarrow H_{orb}^2\left(\mathcal P\left(w_0,w_1\right)\right)$$ is the identity map.

On the other hand, we can show that
$$\deg p_1=\text{lcm}\left(w_0,w_1\right)=w_1,$$
by direct computations. Therefore, we have $$\deg p=\left\|{\bf w}\right\|.$$
\end{proof}

In order to figure out the orbifold canonical line bundle $\mathcal K^{orb}_{\mathcal O_{\mathcal P\left({\bf w}\right)}}$, we need to consider some line orbibundles $\mathcal O_{\mathcal P\left({\bf w}\right)}$ over $\mathcal P\left({\bf w}\right)$. The sheaves $\mathcal O_{\mathbb {CP}\left({\bf w}\right)}(m)$ on $\mathbb {CP}\left({\bf w}\right)$ generated by the graded $S\left({\bf w}\right)$-modules $S\left({\bf w}\right)(m)$ induce the orbisheaves $\mathcal O_{\mathcal P\left({\bf w}\right)}(m)$, which happen to be free hence line orbibundles and satisfy the property 
\begin{equation}\label{orbisheaves}
\mathcal O_{\mathcal P\left({\bf w}\right)}(k)\otimes\mathcal O_{\mathcal P\left({\bf w}\right)}(l)=\mathcal O_{\mathcal P\left({\bf w}\right)}(k+l), ~k,l\in\mathbb Z,
\end{equation}
none of which are true for $\mathcal O_{\mathbb {CP}\left({\bf w}\right)}(m)$ (refer to \cite[Theorem 4.5.4]{BG}) in general.

\begin{prop}\label{wpsprop}
$\mathcal P\left({\bf w}\right)$ satisfies the conditions in Theorem \ref{mainthm}. Also, we have
$$\mu_P\left(\mathcal P\left({\bf w}\right)\right)=2~\left|{\bf w}\right|.$$
\end{prop}

\begin{proof}
Since the line orbibundle generated by $\left[\omega\right]$ in $\mathcal P\left({\bf w}\right)$ is actually $\mathcal O_{\mathcal P\left({\bf w}\right)}(-1)$, it is sufficient to compare $\mathcal O_{\mathcal P\left({\bf w}\right)}(-1)$ with $\mathcal K^{orb}_{\mathcal P\left({\bf w}\right)}$. Plus, it is known that the canonical divisor $K_{\mathbb P\left({\bf w}\right)}$ of the base space, which induces the canonical orbidivisor $\mathcal K^{orb}_{\mathcal P\left({\bf w}\right)}$, is equal to $\mathcal O_{\mathbb P\left({\bf w}\right)}\left(-|{\bf w}|\right)$. Because of (\ref{orbisheaves}), it is clear that
$$c_1^{orb}\left(\mathcal P\left({\bf w}\right)\right)=-\left|{\bf w}\right|\times c_1^{orb}\left(\mathcal O_{\mathcal P\left({\bf w}\right)}\left(-1\right)\right)=-b_{\mathcal P\left({\bf w}\right)}c_1^{orb}\left(\mathcal O_{\mathcal P\left({\bf w}\right)}\left(-1\right)\right)=-b_{\mathcal P\left({\bf w}\right)}\left[\omega\right].$$ Therefore, $b_{\mathcal P\left({\bf w}\right)}=\left|{\bf w}\right|$.
\end{proof}

Any quasi-smooth weighted complete intersection $X$ in $\mathbb P\left({\bf w}\right)$ has an orbifold structure naturally induced by $\mathcal P\left({\bf w}\right)$ (\cite[Proposition 4.6.6]{BG}). Denote $X$ with this orbifold structure by $\mathcal X$.

Also, it is known that the link $$L_X=S^{2n+1}\cap CX$$ with $\dim\geq2$ is simply connected \cite[3.2.12]{D}. Since the locally free action of $S^1$ on $L_X$ produces the orbifold $\mathcal X$, and $L_X$ is a manifold in this case, we see that
$$\pi_1^{orb}(\mathcal X)=0,$$ by the exact sequence (\ref{exactseq}).

Clearly, $\mathcal X$ is a symplectic orbifold with the inherited symplectic form $\iota^\ast\omega$ from $\mathcal P({\bf w})$.

\begin{prop}
Let $\mathcal X$ be a quasi-smooth weighted complete intersection of $\mathcal P({\bf w})$, whose degree is $$\left(m_1,m_2,\cdots m_r\right)$$
with $1\leq r\leq n-2.$ Then, we have
$$\mu_{P}\left(\mathcal X\right)=2\left(\left|{\bf w}\right|-\sum_{j=1}^rm_j\right).$$
\end{prop}

\begin{proof}
Owing to the condition $r\leq n-2$, we have $\dim X\geq 2,$ and hence $\pi_1^{orb}(\mathcal X)=0$.
Consider the injective map induced by the embedding
$$\iota^\ast:H^2\left(\mathbb P\left({\bf w}\right),\mathbb Q\right)\cong\mathbb Q\longrightarrow H^2\left(X,\mathbb Q\right)$$ \cite[(B22)Theorem]{D}.

Note that the local uniformizing chart $\tilde U_j$ of $\mathcal P\left({\bf w}\right)$ is equal to $Spec\left(k\left[z_0,\cdots,\widehat{z_j},\cdots,z_n\right]\right)$, ($k=\mathbb C$), and that $U_j=Spec\left(k\left[z_0,\cdots,z_n\right]_{\left(z_j\right)}\right)$ is the quotient space $\tilde U_j\slash\mathbb Z_{w_j}$ (\cite[Theorem 3A.1.]{BR}). Also $\mathcal O_{\mathcal P\left({\bf w}\right)}(m)|_{\tilde U_j}$ is generated by $\left(\sqrt[w_j]{z_j}\right)^m$ over $\mathcal O_{\tilde U_j}$ (\cite[Theorem 4.5.4]{BG}), and $W_j=C_X\cap\tilde U_j$ is a smooth affine variety for the affine quasicone $C_X$ of $X$ (\cite[\S3.1]{Dol}). 

Consequently, we may use $W_j$ as a local uniformizing chart for the orbifold $\mathcal X$, on which $\mathcal O_{\mathcal X}(m)|_{W_j}$ is equal to $\mathcal O_{\mathcal X}|_{W_j}\otimes_{\mathcal O_{\tilde U_j}}\mathcal O_{\mathcal P\left({\bf w}\right)}(m)|_{\tilde U_j}$, a sheaf of module generated by  $\left(\sqrt[w_j]{z_j}\right)^m$ over $\mathcal O_{\mathcal X}|_{W_j}$. Therefore, 
 $\mathcal O_{\mathcal X}(m)$ is an invertible orbisheaf and $$\mathcal O_{\mathcal X}(m)\otimes\mathcal O_{\mathcal X}(l)=\mathcal O_{\mathcal X}(m+l)$$
still holds. 
Because it is now clear that the line bundle generated by $\left[\iota^{\ast}\omega\right]$ is equal to $\mathcal O_{\mathcal X}(-1)$, and also we know that $K_{X}=\mathcal O_{X}\left(\sum_{j=1}^rm_j-|{\bf w}|\right)$ (\cite[3.3.4 Theorem]{Dol}), we may let
$c^{orb}_1\left(-\mathcal K^{orb}_{\mathcal X}\right)=-b_{\mathcal Z}\left[\iota^\ast\omega\right]=-\iota^\ast\left(b_{\mathcal Z}\left[\omega\right]\right),$
so that
$$-b_{\mathcal Z}\left[\omega\right]=-b_{\mathcal Z}c^{orb}_1\left(\mathcal O_{\mathcal P\left(\bf w\right)}\left(-1\right)\right)=c^{orb}_1\left(\mathcal O_{\mathcal P\left(\bf w\right)}\left(\left|\bf w\right|-\sum_{j=1}^rm_j\right)\right),$$
which leads us to the conclusion.
\end{proof}

Now, turn our attention to the Brieskorn polynomial of the exponent vector
$${\bf a}=\left(a_0,\cdots,a_n\right), n\geq3.$$
i.e.
$$\sum_{j=0}^n z_j^{a_j}=0.$$
Write $l=\text{lcm}_j\{a_j\}$.
Then, we may regard it as a quasi-smooth weighted complete intersection $\mathcal X$ of a single weighted homogeneous polynomial of degree $l$ in the weighted projective space with weights $w_j=l \slash a_j$, and hence it has the orbifold structure induced by the standard weighted projective space (refer to \cite[example 4.6.7]{BG}).

Here, let us define $l_2$ in the following way.
First, let $l=p_1^{s_1}\cdots p_t^{s_t}$
be the prime factorization of $l$.
For each prime $p_j$, choose $a_{j_p}$ whose $\beta_j=\text{ord}_{p}(a_{j_p})$
is the second largest among $\text{ord}_{p}a_{\alpha}$'s. Put
$l_2 = p_1^{\beta_1}\cdots p_t^{\beta_t}$.

\begin{Rmk}
For ${\bf a}=\left(a,a,\cdots,a\right) (n\geq2)$ for example, we have $l=l_2=a$.
If ${\bf a}=\left(a_0,\cdots,a_n\right)$ consists of pairwise relatively prime integers, then it is clear that
$$l=a_0\times\cdots\times a_n, \qquad l_2=1.$$
More concretely, for ${\bf a}=\left(p,p^2,p^3\right)$ with $p$ prime, we have $l=p^3,~l_2=p^2$.
\end{Rmk}

\begin{lem}
Under the above assumption, we have
$$a_{\bf w}=l\slash l_2,$$
and hence $\mathcal X$ is well-formed iff $l=l_2$.
\end{lem}

\begin{proof}
Note that each $w_j$ divides $l$ and so do all $d_j$'s, which leads us to $a_{\bf w} | l$.
In order to get $a_{\bf w}$, we need the biggest $\text{ord}_{p}$ among $d_j$'s, for each prime factor $p$ of $l$.
For each $p_j$, let $\alpha_j$ be the index where
$\text{ord}_{p_j}\left(a_{\alpha_j}\right)=s_j$.
Then $p_j$ does not divide $w_{\alpha_j}$, and hence $d_{\alpha}$ is indivisible by $p_j$ for all $\alpha$ different from $\alpha_j$.
Among $w_{\alpha}$ with $\alpha\neq\alpha_j$, $s_j-\beta_j$ is the least one of $\text{ord}_{p}\left(w_{\alpha}\right).$ i.e. $$\text{ord}_{p}\left(d_{\alpha_j}\right)=s_j-\beta_j,$$ and hence
$$\text{ord}_{p}\left(a_{\bf w}\right)=s_j-\beta_j.$$

Because of $l=l_2$, the ambient projective space is well-formed and so is its complete intersection of the Brieskorn hypersurface by the criterion introduced in Corollary 4.6.10 in \cite{BG}.
\end{proof}

\begin{cor}
Consider the Brieskorn orbifold with the exponent vector
$${\bf a}=\left(a_0,\cdots,a_n\right), n\geq3$$
as a weighted hypersurface in the weighted projective space with weights $\left\{l\slash a_j \right\}$. Then the Conley-Zehnder index of the principal orbit is
$$\mu_P=2~l\left(\sum_{j=0}^n\frac{1}{a_j}-1\right).$$
\end{cor}
\begin{proof}
It follows directly from the preceding description.
\end{proof}
\begin{Rmk}
This value matches (\ref{OldBrieskorn}) although their paths of symplectomorphisms are not defined exactly in the same manner.
\end{Rmk}

\section{Some computations for non-principal orbits}
In this section, we will see how Theorem \ref{mainthm} works for non-principal orbits as well by tackling the examples in the previous section. To begin with, we need to ponder over the Conley-Zehnder index a bit more for a path of unitary matrices. In the simplest case, the path is nothing but complex-valued with modulus 1 and the index goes as follows:
$$\mu_{CZ}\left(e^{\pi it}|_{t\in[0,T]}\right)=\begin{cases}T&\text{if }T\in 2\mathbb Z,\\
2\left\lfloor T\slash2\right\rfloor+1&\text{otherwise.}\end{cases}$$
At times, the product property
$$\mu_{CZ}\left(\Psi^\prime\oplus\Psi^{\prime\prime}\right)=\mu_{CZ}\left(\Psi^\prime\right)+\mu_{CZ}\left(\Psi^{\prime\prime}\right)$$
helps us with higher dimensional cases (refer to \cite[\S2.4]{S}).

Let us begin with the one-dimensional weighted projective space $\mathcal P\left(m,n\right)$ ($m,n$ are relatively prime). Consider the non-principal orbit
$$\gamma(t)=\left(e^{2\pi it},0\right),t\in[0,1],$$
in $\mathcal P(m,n)$. Clearly, it corresponds to the point $x$ whose isotropy group is $\Gamma_x=\mathbb Z_m$.
By Theorem \ref{mainthm}, $\mu_{CZ}(m\cdot\gamma)=2(m+n)$, and it is obvious that 
$$\mu_{CZ}(\gamma)=2\left\lfloor\frac{m+n}{2m}\right\rfloor+1,$$
because it is a path of one-dimensional unitary matrices.

Now, move on to the higher dimensional weighted projective space $\mathcal P\left({\bf w}\right)$ with weights ${\bf w}=\left(w_0,w_1,\cdots,w_n\right)$, and with Reeb vector field $\xi_{\bf w}=2\pi i\sum_{j=0}^{n} w_j\left( z_j\frac{\partial}{\partial z_j}-\bar z_j\frac{\partial}{\partial \bar z_j}\right)$. Since $\mathcal P\left({\bf w}\right)$ has a compatible hermitian metric, the action by $\xi_{\bf w}$ along a principal orbit can be represented as a loop of unitary matrices $A(t)$. Therefore, if we consider another vector field $\xi_{\bf w_0}=2\pi i\sum_{j=0}^{n-1} w_j\left( z_j\frac{\partial}{\partial z_j}-\bar z_j\frac{\partial}{\partial \bar z_j}\right)$ along the same orbit in $\mathcal P\left({\bf w}\right)$, then $A$ can be written as
$$A(t)=\begin{bmatrix}A_0(t)&\ast\\0&\varphi(t)\end{bmatrix},$$ where $A_0(t)$ is generated by $\xi_{\bf w_0}$ and $\varphi(t)$ is a $\mathbb C$-valued function. Also, because $A(t), A_0(t)$ are unitary matrices, $$1=\left|\det A(t)\right|=\left|\det A_0(t)\right|\left|\varphi(t)\right|=\left|\varphi(t)\right|.$$

Consider the non-principal orbit $\gamma_0=\left(e^{2\pi w_0it}z_0,\cdots,e^{2\pi w_{n-1}it}z_{n-1},0\right)$ in $\mathcal P\left({\bf w}\right)$ with nonzero $z_{j}$'s, which corresponds to the point whose isotropy group is $\mathbb Z_{d_n}$, where $d_n=\gcd\left(w_0,w_1,\cdots,w_{n-1}\right)$. By applying Proposition \ref{wpsprop} to $\mathcal P\left({\bf w}\right)$ and to one-dimensional lower weighted projective space $\mathcal P\left({\bf w}_0\right)$ with weights ${\bf w_0}=\left(w_0,w_1,\cdots,w_{n-1}\right)\slash d_n$, we know that $\varphi(t)$ contributes twice as much as $\left|{\bf w}\right|-d_n\cdot\left|{\bf w}_0\right|=w_n$ to the Conley-Zehnder index in $\mathcal P\left({\bf w}\right)$ along $\gamma_0$ traveling $d_n$-times repeatedly. Therefore, we get
$$\mu_{CZ}\left(\gamma_0\right)=\frac{2}{d_n}\sum_{j=0}^{n-1}w_j+2\left\lfloor \frac{w_n}{2d_n}\right\rfloor+1.$$
By using this method repeatedly, we can get the Conley-Zehnder index for any non-principal orbits in the weighted projective spaces. For example, consider 
$\gamma_0=\left(e^{2\pi it},e^{2\pi it},0,0\right),$ $t\in\left[0,1\right]$, in $\mathcal P(4,4,5,14)$. With an iteration of the above operation, we get
$$\mu_{CZ}\left(\gamma_0\right)=2\cdot(1+1)+(2\cdot\left\lfloor\frac{5}{2\cdot4}\right\rfloor+1)+(2\cdot\left\lfloor\frac{7}{2\cdot2}\right\rfloor+1)=8.$$
after all.

It's even possible to apply this method for the Brieskorn cases. Consider
$\gamma_0=\left(e^{\pi it}z_0,e^{\pi it}z_1,e^{\pi it}z_2,0\right),$ $t\in\left[0,1\right]$ in $z_0^{2}+z_1^{2}+z_2^{2}+z_3^{5}=0$, ($z_0z_1z_2\neq0$). Note that it is a quasi-smooth hypersurface with degree $d=\text{lcm}(2,2,2,5)=10$ in $\mathcal P(5,5,5,2)$, and $\gamma_0$ corresponds to a point with an isotropy group of order 5 when counting it with the induced orbifold structure. Therefore, if we make use of the similar method as above, we get
$$\mu_{CZ}\left(\gamma_0\right)=2\cdot(1+1+1-2)+(2\cdot\left\lfloor\frac{2}{2\cdot5}\right\rfloor+1)=3,$$
because of $\mu_{P}=2\cdot\left(5+5+5+2-10\right).$

\begin{Rmk}
We need to be careful in applying this method to non-principal orbits in general orbifolds, because we have implicitly used the fact that in $\mathbb C^n$, the subspaces $\{z_n=0\}$ and $\{z_1=\cdots=z_{n-1}=0\}$ are unitarily complementary to each other so that they are not only orthogonal but also symplectic complement of each other. Apparently, this is not always true for general orbifold strata.
\end{Rmk}

\bigskip


\bigskip

\end{document}